\documentclass [12pt]{article}
\usepackage{amssymb,amsmath,comment,amsthm}

\def\N{\mathbb{N}}

\def\F{\mathbb{F}}

\newtheorem{theorem}{Theorem}[section]
\newtheorem{proposition}[theorem]{Proposition}
\newtheorem{corollary}[theorem]{Corollary}
\newtheorem{lemma}[theorem]{Lemma}

\newtheorem{remark}[theorem]{Remark}
\newtheorem{remarks}[theorem]{Remarks}
\newtheorem{conjecture}[theorem]{Conjecture}
\newtheorem{notation}[theorem]{Notation}
\newtheorem{notations}[theorem]{Notations}

\begin{document}
\title{On (unitary) perfect polynomials over $\F_2$ with only Mersenne primes as odd divisors}
\author{Gallardo Luis H. - Rahavandrainy Olivier  \\
Universit\'e de Brest, UMR CNRS 6205\\
Laboratoire de Math\'ematiques de Bretagne Atlantique\\
e-mail : luisgall@univ-brest.fr - rahavand@univ-brest.fr}
\maketitle
\begin{itemize}
\item[a)]
Running head: Mersenne and perfect polynomials
\item[b)]
\item[c)]
\item[d)]
Corresponding author:
\begin{center} Luis H. Gallardo
\end{center}
\end{itemize}

\newpage~\\
{\bf{Abstract}}\\
The only (unitary) perfect polynomials over $\F_2$  that are products of $x$, $x+1$ and  Mersenne primes are precisely the nine (resp. nine ``classes'') known ones. This follows from a new result about the factorization of $M^{2h+1} +1$, for a Mersenne prime $M$ and for a positive integer $h$. Other consequences of such a factorization are new results about odd perfect polynomials.

{\section{Introduction}}
Let $A \in \F_2[x]$ be a nonzero polynomial. We say that $A$ is \emph{even} if it has a linear factor and that  it is
\emph{odd} otherwise. We define a
\emph{Mersenne prime} (polynomial) over $\F_2$ as an irreducible polynomial of the form
$1+x^a(x+1)^b$, for some positive integers $a,b$. This comes as an analogue of the prime factors of the even perfect numbers.
As over the integers, we say that a divisor $d$ of $A$ is \emph{unitary} if $\displaystyle{\gcd(d,\frac{A}{d}) = 1}$.
Let $\omega(A)$ denote the number of distinct irreducible (or
\emph{prime}) factors of $A$ over $\F_2$ and let $\sigma(A)$ (resp. $\sigma^*(A)$) denote the sum of all (unitary) divisors of
$A$ (both $\sigma$ and $\sigma^*$ are multiplicative
functions). If $\sigma(A) = A$ (resp. $\sigma^*(A)=A$), then we
say that $A$ is (\emph{unitary}) \emph{perfect}. Finally, we say that a (unitary) perfect polynomial is \emph{indecomposable} if it is not a product of two coprime nonconstant (unitary) perfect polynomials.

The notion of (unitary) perfect polynomials is introduced in \cite{Canaday} (a simplified version of this Ph. D. thesis under Carlitz)
by E. F. Canaday in $1941$ and extended by J. T. B. Beard
Jr. et al. (probably still inspired by Carlitz that advised the advisor of Beard) in several directions (\cite{Beard2}, \cite{Beard}). Later research in the subject (\cite{Gall-Rahav4}, \cite{Gall-Rahav7},
\cite{Gall-Rahav5}, \cite{Gall-Rahav8},
\cite{Gall-Rahav2}) allows us to more precisely describe such polynomials. For the perfect case, we get:\\
- the ``trivial'' ones, of the form
$(x^2+x)^{2^n-1}$, for some positive integer $n$,\\
- nine others which are the unique even all whose odd factors are Mersenne primes raised to powers of the form $2^n-1$ (\cite{Gall-Rahav12}, Theorem 1.1),\\
- and the last two which are divisible by a non Mersenne prime.\\
By analogy, since we can also consider perfect polynomials, $A \in \F_2[x]$ with $\sigma(A)/A=1$,  as an analogue of multiperfect numbers, $n \in \N^{*}$ avec $\sigma(n)/n \in \N^{*}$,
it might have some interest to observe  that most known multiperfect numbers  (see OEIS sequence A007691) appear to be divisible by a Fermat prime or by a Mersenne prime.\\
Obviously,  all unitary perfect polynomials are even. We prove for the unitary case that essentially, the known ones belong to the nine ``classes''  relative to the equivalence relation :  two unitary perfect polynomials
are equivalent if and only if some power of $2$ of one equals some power of $2$ of the other (see below).

The paper consists of two major results that we describe now. The most important is Theorem \ref{result3} 
that  improves significantly on  these results (because, now there are no conditions on the powers of the $M_j$'s).
Its proof is obtained from new results given in Theorem \ref{result0} which in turn, extends  recent  non-trivial results in  \cite[ Theorem 1.4 ]{Gall-Rahav-mersenn}.

We began to study odd perfect polynomials in \cite{Gall-Rahav4}. They are all squares \cite{Canaday} and must have  \cite{Gall-Rahav7} at least five distinct prime divisors. We have also considered  \cite{Gall-Rahav4} ``special perfect'' polynomials which are of the form $S=P_1^2 \cdots P_m^2$, with each $P_j$ odd and irreducible.
We proved \cite{Gall-Rahav4} that if such a polynomial $S$ is perfect, then $\omega(S) \geq 10$, $\min_j{\deg(P_j)} \geq 30$ and $P_j \equiv 1 \mod x^2+x+1$. We get a new result for them  as well as a  new result about the existence of more general odd perfect polynomials,  in Theorem \ref{caseodd}, as a consequence of Theorem \ref{result0}.\\
Observe that Theorem \ref{result0} is a new step on the proof of a Conjecture about Mersenne primes that is discussed in the recent paper \cite{Gall-Rahav-mersenn}.\\

It is convenient to fix some notations:
\begin{itemize}
\item[(a)]
For $S \in \F_2[x]$, we denote by
$\overline{S}$ the polynomial obtained from $S$ with $x$ replaced by
$x+1$: $\overline{S}(x) = S(x+1)$.
\item[(b)] $\N$ (resp. $\N\sp{*}$) denotes, as usual, the set of nonnegative
integers (resp. of positive integers).
\item[(c)] To avoid trivialities, we suppose that any (unitary) perfect polynomial is indecomposable.
\end{itemize}
\begin{notations}~\\
Set $M_j:= 1+x(x+1)^j, j \in \{1,2,3\}$, ${\mathcal{M}}:= \{M_1, M_2,\overline{M_2}, M_3, \overline{M_3} \}$,\\ ${\mathcal{P}}:= \{T_1,\ldots, T_9\}$ and ${\mathcal{P}}_u:= \{U_1,\ldots, U_9\}$ where:
$$\begin{array}{l}
T_1 =x^2(x+1)M_1, T_2=\overline{T_1}, \\
T_3 = x^4(x+1)^3M_3, T_4 =\overline{T_3},\
T_5 = x^4(x+1)^4M_3\overline{M_3} = \overline{T_5},\\
T_6 = x^6(x+1)^3M_2\overline{M_2}, T_7= \overline{T_6},\\
T_8 = x^4(x+1)^6M_2\overline{M_2} M_3 \text{ and } T_9 = \overline{T_8},\\
U_1 = x^3(x+1)^3M_1^2, \ U_2 = x^3(x+1)^2M_1, \ U_3 = x^5(x+1)^4M_3, \\
U_4 = x^7(x+1)^4M_2 \overline{M_2}, \ U_5 = x^5(x+1)^6M_1^2M_3, \ U_6 = x^5(x+1)^5M_3 \overline{M_3}, \\
U_7 = x^7(x+1)^7M_2^2 {\overline{M_2}}^2, \ U_8 = x^7(x+1)^6M_1^2 M_2 \overline{M_2}, \ U_9 = x^7(x+1)^5M_2 \overline{M_2} \ \overline{M_3}.\\
\end{array}$$
The nine nontrivial perfect polynomials cited above are: $T_1,\ldots, T_9$ and the two others are: $T_{10} = x^2(x+1)(x^4+x+1){M_1}^2$, $T_{11} = \overline{T_{10}}$.\\
The known unitary perfects are all of the form $B^{2^n}$, where $n\in \N$ and $B \in \{U_1,\ldots, U_9\}$.
\end{notations}

Our results are:
\begin{theorem} \label{result3}
Let $A = x^a(x+1)^b \prod_i P_i^{h_i} \in \F_2[x]$ with each $P_i$ Mersenne prime and $h_i \in \N^{*}$. Then $A$ is even $(unitary)$ perfect if and only if $A \in {\mathcal{P}}$ $(resp.  \ A = B^{2^n}$ with $n \in \N$ and $B \in {\mathcal{P}}_u)$.
\end{theorem}

\begin{theorem} \label{caseodd}
i) There exists no special perfect polynomial divisible only by Mersenne primes.\\
ii) If $A=P_1^{2h_1} \cdots P_m^{2h_m}$, where each $P_j$ is a Mersenne prime and if for some $j$, $2h_j+1$ is divisible by a Mersenne prime $\not= 7$ or by a Fermat prime $\not= 5$, then $A$ is not perfect.
\end{theorem}

\begin{theorem} \label{result0}
Let $h$ be a positive integer and $M \in \F_2[x]$ a Mersenne prime. Then, in the following cases, $\sigma(M^{2h})$ is divisible by a non Mersenne prime:\\
i) $(M \in \{M_1, M_3, \overline{M_3} \})$ or $(M \in \{M_2, \overline{M_2}\}$ and $h \geq 2)$.\\
ii) $M \not\in {\mathcal{M}}$ and $2h+1$ is divisible by a prime number $p$, where $(p \not= 7$ is a Mersenne number$)$ or $($the order of $2$ modulo $p$ is divisible by $8$ $($in particular, when $p$ is a Fermat prime greater than $5))$.
\end{theorem}

\section{Proofs of Theorems \ref{result3} and \ref{caseodd}} \label{proof3}
Sufficiency in Theorem \ref{result3} is obtained by direct computations. For the necessity, we shall apply Lemmas \ref{oldresult1} and \ref{oldresult2}, Propositions \ref{caseperfect} and \ref{caseunitperfect}. We use Theorem \ref{result0} to prove these two propositions. A similar method gives Theorem \ref{caseodd}.
We recall below \cite[ Theorem 1.4 ]{Gall-Rahav-mersenn} which partially solves \cite[Conjecture 5.2]{Gall-Rahav12} about the factorization of $\sigma(M^{2h})$:

\begin{conjecture} [Conjecture 5.2 in \cite{Gall-Rahav12}] \label{oldconject}
Let $h \in \N^*$ and $M$ be a Mersenne prime over $\F_2$ such that $M \not\in \{M_2, \overline{M_2}\}$. Then, the polynomial $\sigma(M^{2h})$ is divisible by a non Mersenne prime.
\end{conjecture}

\begin{lemma} [Theorem 1.4 in \cite{Gall-Rahav-mersenn}] \label{oldresult3}
Let $h\in \N^*$ such that $p=2h+1$ is prime, $M$ a Mersenne prime such that $M \not\in \{M_2, \overline{M_2}\}$ and $\omega(\sigma(M^{2h})) = 2$. Then, $\sigma(M^{2h})$ is divisible by a non Mersenne prime.
\end{lemma}

\subsection{Proof of Theorem \ref{result3}} \label{caseven}

We set $\displaystyle{A:= x^a(x+1)^b \prod_{i \in I}  P_i^{h_i} = A_1 A_2 \in \F_2[x]}$, where $a,b, h_i \in \N$, $P_i$ is a Mersenne prime, $\text{$\displaystyle{A_1 = x^a(x+1)^b \prod_{P_i \in {\mathcal{M}}} P_i^{h_i}}$  and $\displaystyle{A_2 = \prod_{P_j \not\in {\mathcal{M}}} P_j^{h_j}}$.}$\\
 We suppose that $A$ is indecomposable (unitary) perfect.
\subsubsection{Case of perfect polynomials} \label{case-perfect}

\begin{lemma} [Theorem 1.1 in \cite{Gall-Rahav12}] \label{oldresult1}
If $h_i = 2^{n_i}-1$ for any $i \in I$, then $A \in {\mathcal{P}}$.
\end{lemma}

We get from Theorem 8 in \cite{Canaday} and from Theorem \ref{result0}:
\begin{lemma} \label{canadayperf}
i) If $\sigma(x^a)$ is divisible only by Mersenne primes, then $a \in \{2,4,6\}$ and all its divisors lie in ${\mathcal{M}}$.\\
ii) Let $M \in {\mathcal{M}}$ such that $\sigma(M^a)$ is divisible only by Mersenne primes, then $a=2$ and $M \in \{M_2, \overline{M_2}\}$.
\end{lemma}

\begin{lemma} \label{divisorsigmA1}
If $P$ is a Mersenne prime divisor of $\sigma(A_1)$, then $P, \overline{P} \in \{M_1,M_2, M_3\}$.
\end{lemma}
\begin{proof}
We apply Lemma \ref{canadayperf}. If $P$ divides $\sigma(x^a) \cdot \sigma((x+1)^b)$, then $P \in {\mathcal{M}}$. If $P$ divides $\sigma(P_i^{h_i})$ with $P_i \in {\mathcal{M}}$, then $P_i \in \{M_2, \overline{M_2} \}$ and $P, \overline{P} \in \{M_1, M_3\}$.
\end{proof}

\begin{lemma} \label{gcdMjsigmA1}
i) For any $P_j \not\in {\mathcal{M}}$, one has: $\gcd(P_j^{h_j},\sigma(A_1)) = 1$ and $h_j = 0$.
ii) $A=A_1$.
\end{lemma}
\begin{proof}
i): Let $P_j \not\in {\mathcal{M}}$ and $Q_i \in {\mathcal{M}}$. Then, $P_j$ divides neither $\sigma(x^a)$, $\sigma((x+1)^b)$ nor $\sigma(Q_i^{h_i})$. Thus $\gcd(P_j^{h_j}, \sigma(A_1)) = 1$. \\
Observe that $P_j^{h_j}$ divides $\sigma(A_2)$ because $P_j^{h_j}$ divides $A = \sigma(A)=\sigma(A_1) \sigma(A_2)$. Hence, $A_2$ divides $\sigma(A_2)$. So, $A_2$ is perfect and it is equal to $1$ from the indecomposibility of $A$.\\
ii) follows from i).
\end{proof}

\begin{proposition} \label{caseperfect}
If $A_1$ is perfect, then $h_j=2^{n_j}-1$ for any $P_j \in {\mathcal{M}}$.
\end{proposition}
\begin{proof}
i) Suppose that $P_j \not\in  \{M_2, \overline{M_2}\}$. If $h_j$ is even, then $\sigma(P_j^{h_j})$ is divisible by a non Mersenne prime $Q$. So, we get the contradiction: $Q \mid A$. If $hj=2^{n_j}u_j-1$ with $u_j \geq 3$ odd, then
$\sigma(P_j^{h_j}) = (1+P_j)^{2^{n_j}-1}\cdot (1+P_j+\cdots + P_j^{u_j-1})^{2^{n_j}}$ is also divisible by a non Mersenne prime, which is impossible.\\
ii) If $P_j \in \{M_2,\overline{M_2} \}$ and ($h_j$ is even or it is of the form $2^{n_j}u_j-1$, with $u_j \geq 3$ odd), then $a,b \in \{7 \cdot 2^n-1: n \geq 0\}$. Thus, for some $\nu \in \N^*$, $M_1^{2^{\nu}}$ divides $\sigma(A)=A$. It is impossible by the part i) of our proof.
\end{proof}
Lemma \ref{gcdMjsigmA1}, Proposition \ref{caseperfect} and Lemma \ref{oldresult1} imply
\begin{corollary} \label{A1perfect}
One has: $A=A_1 \in {\mathcal{P}}$.
\end{corollary}

\subsubsection{Case of unitary perfect polynomials}

Similar proofs give Proposition \ref{caseunitperfect} and thus, our result.

\begin{lemma} [Theorem 1.3 in \cite{Gall-Rahav12}] \label{oldresult2}
 If $h_i = 2^{n_i}$ for any $i \in I$, then $A$ (or $\overline{A}$) is of the form $B^{2^n}$ where $B \in {\mathcal{P}}_u$.
\end{lemma}

\begin{proposition} \label{caseunitperfect}
i) If $A_1$ is unitary perfect then $h_j=2^{n_j}$ for any $P_j \in {\mathcal{M}}$.\\
ii) $A=A_1$.
\end{proposition}

\begin{remark}
\emph{Contrary to our proofs  in the present paper, the proofs of \cite[Corollaries 5.3 and 5.4 ]{Gall-Rahav12} are not complete,  since  the special case where
$\gcd(M_2 \overline{M_2}, A) \neq 1$  was not considered}.

\end{remark}
\subsection{Proof of Theorem \ref{caseodd}} \label{oddperfect}
We also use in this proof  Lemma \ref{oldresult3} and Theorem~\ref{result0}. \\
If  $S=P_1^2 \cdots P_m^2$ is perfect, where each $P_j$ is Mersenne, then
$\sigma(P_1^2) \cdots \sigma(P_m^2) = \sigma(S)=S=P_1^2 \cdots P_m^2$. We must have: $\sigma(P_1^2) = \prod_k Q_k$.Thus, $P_1 \in \{M_2, \overline{M_2} \}$ and any $Q_k \in \{M_1, M_3, \overline{M_3} \}$. But, for any $T \in \{M_1, M_3, \overline{M_3} \}$, $\sigma(T^2)$ and thus $S$ is divisible by a non Mersenne prime, which is impossible.\\
We get in the same manner the part ii) of the theorem.

\section{Proof of Theorem \ref{result0}} \label{proof0}
We mainly prove Theorem \ref{result0} by contradiction (to Corollary \ref{sommaieven}). Lemma \ref{omegasigmM2h} states that $\sigma(M^{2h})$ is square-free for any $h \in \N^*$. We suppose that:
\begin{equation}\label{assume}
\text{$\sigma(M^{2h})= \displaystyle{\prod_{j\in J} {P_j}}$, $P_j = 1+x^{a_j}(x+1)^{b_j}$ irreducible, $P_i \not= P_j$ if $i \not= j$.}
\end{equation}
We set $U_{2h}:=\sigma(\sigma(M^{2h}))$ and $M:=x^a(x+1)^b+1$, with $M$ irreducible (so that $\gcd(a,b)=1$, $a$ or $b$ is odd). We may assume that $a$ is odd, without loss of generality.

\subsection{Useful facts}\label{preliminaire}
Some of the following results are obvious or cited in \cite{Gall-Rahav-mersenn}, so we omit their proofs.
By Lemma \ref{p-reduction}, $\sigma(M^{2h})$ is divisible by a non Mersenne prime whenever $\sigma(M^{p-1})$  is too, 
for some prime divisor $p$ of $2h+1$.

\begin{lemma} \label{phiandirreduc}
For $m\in \N^*$, denote, as usual, by $N_2(m)$ the number of irreducible polynomials of degree $m$, over $\F_2$. Then\\
i) $N_2(m) \geq [2^m - 2(2^{m/2} - 1)]/m$.\\
ii) $\varphi(m) < N_2(m)$ if $m\geq 4$, where $\varphi$ is the Euler totient function.\\
iii) There exists at least one irreducible polynomial of degree $m$ which is not a Mersenne prime, if $m\geq 4$.
\end{lemma}

\begin{proof}
i): See Exercise 3.27, p. 142 in \cite{Rudolf}.\\
ii): We get by direct computations, $m < (2^m - 2(2^{m/2} - 1))/m$ for $4 \leq m \leq 5$ and by studying the function $f(x)=2^x - 2(2^{x/2} - 1) -x^2$, for $x\geq 6$. So,
$\varphi(m) \leq m< N_2(m)$.\\
iii): First, $1+x^c(x+1)^d \text{ Mersenne prime}$ implies that $\gcd(c,c+d)=\gcd(c,d)=1$. Moreover,
the set ${\mathcal{M}}_m$ of Mersenne primes of degree $m$ is a subset of $\Sigma_m:=\{ x^c(x+1)^{m-c} +1: 1 \leq c \leq m, \ \gcd(c,m)=1\},$
Thus,
$$\# {\mathcal{M}}_m \leq \# \{c: 1\leq c \leq m,\ \gcd(c,m)=1\} = \varphi(m).$$
Therefore, there exist at most $\varphi(m)$ Mersenne primes of degree $m$. So, we get iii).
\end{proof}

\begin{lemma} \label{omegasigmM2h}
i) $\sigma(M^{2h})$ is square-free and reducible.\\
ii) $a \geq 2$ or $b \geq 2$ so that $M \not=M_1$.
\end{lemma}

\begin{notation}
\emph{For a nonconstant polynomial $S$ of degree $s$, we denote by $\alpha_l(S)$ the coefficient of $x^{s-l}$ in $S$, $0\leq l \leq s.$ One has: $\alpha_0(S) =1$.}
\end{notation}

We sometimes apply Lemmas \ref{lesalfal} and \ref{alfasigmM2h} without explicit mentions.

\begin{lemma} \label{lesalfal}
Let $S \in \F_2[x]$ of degree $s \geq 1$ and $l,t,r,r_1,\ldots, r_k \in \N$ such that $r_1>\cdots >r_k$, $t \leq k,
r_1-r_t \leq l \leq r \leq s.$  Then\\
i) $\alpha_l[(x^{r_1} + \cdots +x^{r_k})S] = \alpha_l(S) + \alpha_{l-(r_1-r_2)}(S)+\cdots + \alpha_{l-(r_1-r_t)}(S)$.\\
ii) $\alpha_l(\sigma(S))=\alpha_l(S)$ if no irreducible polynomial of degree at most $r$~divides~$S$.
\end{lemma}

\begin{proof}
i): Obvious, by definition of $\alpha_l$.\\
ii) Follows from the fact: $\sigma(S) = S + T$, where $\deg(T) \leq \deg(S)-r-1$.
\end{proof}

\begin{corollary} \label{sommaieven}
i) The integers $\displaystyle{u=\sum_{j \in J} a_j}$ and $\displaystyle{v=\sum_{j \in J} b_j}$ are both even.\\
ii) $U_{2h}$ splits $($over $\F_2)$.\\
iii) $U_{2h}$ is a square so that $\alpha_k(U_{2h}) = 0$ for any odd positive integer $k$.
\end{corollary}

\begin{proof}
i): See \cite[Corollary 4.9]{Gall-Rahav-mersenn}.\\
ii) and iii): Assumption (\ref{assume}) implies that $$\displaystyle{U_{2h}= \sigma(\sigma(M^{2h}))=\sigma(\prod_{j\in J} {P_j}) =  \prod_{j \in J} x^{a_j}(x+1)^{b_j} = x^u(x+1)^v},$$ with $u$ and $v$ both even.
\end{proof}

\begin{lemma} \label{alfasigmM2h}
One has modulo $2$: $\alpha_l(\sigma(M^{2h}))=\alpha_l(M^{2h})$ if $1\leq l \leq a+b-1$,
$\alpha_l(\sigma(M^{2h})) = \alpha_l(M^{2h}+M^{2h-1})$ if $a+b \leq l \leq 2(a+b)-1$.
\end{lemma}
\begin{proof}
Since $\sigma(M^{2h}) = M^{2h} + M^{2h-1} + T$, with $\deg(T) \leq (a+b)(2h-2)=2h(a+b)-2(a+b)$, Lemma \ref{lesalfal}-ii) implies that
$\alpha_l(\sigma(M^{2h})) = \alpha_l(M^{2h})$ if $1\leq l \leq a+b-1$ and
$\alpha_l(\sigma(M^{2h})) = \alpha_l(M^{2h}+M^{2h-1})$
if $a+b\leq l \leq 2(a+b)-1$.
\end{proof}

Lemma below (with analogous proof) is a generalization of Lemma 4.10 in \cite{Gall-Rahav-mersenn}.

\begin{lemma} \label{p-reduction}
If $k$ divides $2h+1$ $($with $k$ prime or not$)$, then $\sigma(M^{k-1})$ divides $\sigma(M^{2h})$.
\end{lemma}

We fix a prime factor $p$ of $2h+1$. We denote by $ord_p(2)$ the order of $2$ in $\F_p \setminus \{0\}$.

\begin{lemma} \label{ord(2)divise}
For any $j \in J$, $ord_p(2)$ divides $a_j+b_j = \deg(P_j)$.
\end{lemma}

\begin{proof}
Let $d=\gcd_i(a_i+b_i)$. By Lemma 4.13 in \cite{Gall-Rahav-mersenn}, $p$ divides $2^d-1$. Thus, $ord_p(2)$ divides $d$.
\end{proof}

\begin{lemma} \label{reduction2}
Let $P_i = 1+x^{a_i}(x+1)^{b_i}$ be a prime divisor of $\sigma(M^{p-1}),$ where $2^{a_i+b_i}-1 = p_i$ is a prime number.  Then\\
i) any irreducible polynomial $($Mersenne or not$)$ of degree $a_i+b_i$ divides $\sigma(M^{p-1})$.\\
ii) $\sigma(M^{p-1})$ is divisible by a non Mersenne prime if $a_i+b_i \geq 4$.
\end{lemma}
\begin{proof}
First, $P_i$ is a primitive polynomial. Let $\alpha$ be a root of $P_i$. One has $M(\alpha)^p = 1$, $M(\alpha) = \alpha^r$ for some $1\leq r \leq p_i-1$. Thus, $1= M(\alpha)^p = \alpha^{rp}$, with $ord(\alpha) = p_i$. So, $p_i$ divides $rp$ and $p_i=p$.\\
i): If $P$ is an irreducible polynomial of degree $a_i+b_i$, then $P$ is primitive. Let $\beta$ be a root of $P$. One has $ord(\beta) = p_i=p$, $P(\beta)=0$ and $M(\beta) = \beta^s$, for some $1\leq s \leq p_i-1$. Thus, $M(\beta)^{p} = \beta^{ps}= 1$.\\
ii) follows from i) and from Lemma \ref{phiandirreduc}-iii).
\end{proof}
\begin{corollary} \label{reduction3}
For any $i \in J$, $a_i+b_i \leq 3$ or $2^{a_i+b_i} -1$ is not prime.
\end{corollary}

\begin{lemma} \label{PandQ}
Let $P, Q \in \F_2[x]$ such that $\deg(P)=r$, $2^r-1$ is prime, $P \nmid Q(Q+1)$ but $P \mid Q^p +1$. Then $2^r-1=p$.
\end{lemma}
\begin{proof}
Let $\beta$ be a root of $P$. $\beta$ is primitive, $ord(\beta) = 2^r-1$, $Q(\beta) \not\in \{0,1\}$ because $P \nmid Q(Q+1)$. Thus, $Q(\beta) = \beta^t$ for some $1\leq t \leq 2^r-2$. Hence, $1=Q(\beta)^p = \beta^{tp}.$ So, $2^r-1$ divides $tp$ and $2^r-1=p$.
\end{proof}
\begin{corollary} \label{notdivisor}
Let $r \in \N^*$ such that $2^r-1$ is a prime distinct from $p$. Then, no irreducible polynomial of degree $r$ divides $\sigma(M^{p-1})$.
\end{corollary}
\begin{proof}
If $P$ divides $\sigma(M^{p-1})$ with $\deg(P)=r$, then $P$ divides $M^p+1$ and by taking $Q = M$ in the above lemma, we get a contradiction.
\end{proof}

In the following three lemma and corollaries, we suppose that $p$ is a Mersenne prime of the form $2^m-1$ (with $m$ prime).

\begin{lemma} \label{polyPandQ}
Let $P, Q \in \F_2[x]$ such that $P$ is irreducible of degree $m$ and $P \nmid Q(Q+1)$.
Then, $P$ divides $Q^p +1$.
\end{lemma}
\begin{proof}
Let $\beta$ be a root of $P$. $P$ and $\beta$ are primitive, $ord(\beta) = 2^m-1=p$, $Q(\beta) \not\in \{0,1\}$. Thus, $Q(\beta) = \beta^t$ for some $1\leq t \leq p-1$. Hence, $Q(\beta)^p = \beta^{tp} = 1.$ So, $P$ divides $Q^p + 1$.
\end{proof}
\begin{corollary} \label{anydivides}
Any irreducible polynomial $P \not = M$ $($Mersenne or not$)$, of degree $m$, divides $\sigma(M^{p-1})$.
\end{corollary}
\begin{proof}
$P$ does not divide $x^a(x+1)^bM=M(M+1)=Q(Q+1)$. So, we apply Lemma \ref{polyPandQ} to $Q = M$.
\end{proof}
\begin{corollary} \label{reduction4}
The polynomial
$1+x+x^2$ divides $\sigma(M^{p-1})$ if and only if $(M \not= 1+x+x^2$ and $p=3)$,\\
$1+x^2+x^3$ divides $\sigma(M^{p-1})$ if and only if $M \not= 1+x^2+x^3$ and $p=7$,\\
$1+x+x^3$ divides $\sigma(M^{p-1})$ if and only if $M \not= 1+x+x^3$ and $p=7$.
\end{corollary}
\begin{proof} Apply Corollary \ref{anydivides} with $m\in \{2,3\}$.
\end{proof}

\subsection{Case $M \in \{M_1, M_3, \overline{M_3}\}$}\label{M=M123}
Lemma \ref{omegasigmM2h} implies that $M \not= M_1$. It suffices to suppose that $M =M_3$.

We refer to Section 5.2 in \cite{Gall-Rahav13}. Put $U:=M_1M_2\overline{M_2}$. By \cite[ Lemma 5.4]{Gall-Rahav13}, we have to consider four cases:\\
i) $\gcd(\sigma(M^{2h}),U) = 1$, \\
ii) $\sigma(M^{2h}) = M_1 B$, with $\gcd(B,U)=1$,\\
iii) $\sigma(M^{2h}) = M_2\overline{M_2} B$, with $\gcd(B,U)=1$,\\
iv) $\sigma(M^{2h}) = U B$, with $\gcd(B,U)=1$,\\
where any irreducible divisor of $B$ has degree exceeding $5$.\\
We get Lemma below which contradicts the fact that $U_{2h}$ is a square.
\begin{lemma}
$\alpha_3(U_{2h})= 1$ or $\alpha_5(U_{2h})= 1$.
\end{lemma}
\begin{proof}
For i), iii) and iv) : use Lemmas 5.9, 5.10, 5.15, 5.17 (still in \cite{Gall-Rahav13}).\\
For ii): since $\sigma(M^{2h}) = (x^2+x+1)B$ and $U_{2h} =  (x^2+x) \sigma(B)$, we obtain (by Lemmas \ref{lesalfal} and \ref{alfasigmM2h}):
$$\begin{array}{l}
0= \alpha_1(M^{2h}) = \alpha_1(\sigma(M^{2h})) = \alpha_1(B)+1,\\
\alpha_3(U_{2h}) = \alpha_3(\sigma(B)) + \alpha_2(\sigma(B)) = \alpha_3(B)+\alpha_2(B),\\
0 = \alpha_3(M^{2h}) = \alpha_3(\sigma(M^{2h})) = \alpha_3(B)+\alpha_2(B)+\alpha_1(B).
\end{array}
$$
Thus, $\alpha_3(U_{2h}) = \alpha_3(B)+\alpha_2(B) = \alpha_1(B)=1$.
\end{proof}

\subsection{Case where $M \in \{M_2, \overline{M_2}\}$ and $h \geq 2$} \label{caseM2}

It suffices to consider $M=M_2=1+x+x^3$.
Recall that $U_{2h} = \sigma(\sigma(M^{2h}))$ splits and it is a square. Note also that if $h=1$, then $\sigma({M_2}^{2h}) = \sigma({M_2}^{2})= M_1M_3$.\\

For $h \in \{2,3\}$, we get by direct computations, $U_4 = x^3(x+1)^6 (x^3+x+1)$ and
$U_6 = x^8(x+1)^4(x^3+x+1)^2$ which do not split (even if $U_6$ is a square).\\

So, $h \geq 4$.

\begin{lemma} \label{divisordeg}
i) $1+x+x^2$ divides $\sigma(M^{2h})$ if and only if $3$ divides $2h+1$.\\
ii) $1+x^2+x^3$ divides $\sigma(M^{2h})$ if and only if $7$ divides $2h+1$.\\
iii) Any irreducible divisor of $\sigma(M^{2h})$ is of degree at least $4$, if $2h+1$ is divisible by a prime $p \not\in \{3,7\}$.
\end{lemma}
\begin{proof}
i) and ii): from Corollaries \ref{notdivisor} and \ref{anydivides}.\\
iii) follows from i) and ii).
\end{proof}

\subsubsection{Case where $2h+1$ is divisible by a prime $p \not\in \{3,7\}$}
By Lemma \ref{p-reduction}, $\sigma(M^{p-1})$ divides $\sigma(M^{2h})$. So, we may suppose that $2h+1 =p$ so that $2h=p-1$. It suffices then to prove (directly or by a contradiction)  that $\sigma(M^{2h})$ is divisible by a non Mersenne prime.

\begin{lemma}
i) $\alpha_l(U_{2h}) = \alpha_l(\sigma(M^{2h}))$ for $l\in \{1,2,3\}$.\\
ii) $\alpha_l(\sigma(M^{2h})) = \alpha_l(M^{2h})$ for $l\in \{1,2\}$, $\alpha_3(\sigma(M^{2h})) = \alpha_3(M^{2h}+M^{2h-1})$.
\end{lemma}
\begin{proof}
i) follows from Lemma \ref{divisordeg}.\\
ii): for $l \leq 2$, one has: $6h -l = \deg(\sigma(M^{2h}))-l = \deg((M^{2h})-l > 3(2h-1) = \deg(M^{2h-1})$ and for $3 \leq l\leq 5$, $6h -l > 3(2h-2) = \deg(M^{2h-2})$. Hence, we get ii).
\end{proof}
\begin{corollary}
$\alpha_3(U_{2h}) = 1$ if $h \geq 4$.
\end{corollary}
\begin{proof}
$\alpha_3(U_{2h}) = \alpha_3(M^{2h}+M^{2h-1}) = \alpha_3[(x^3+x)M^{2h-1}] = \alpha_3(M^{2h-1}) + \alpha_1(M^{2h-1})$.
But, $M^{2h-1} = (x^3+x+1)^{2h-1} = (x^3+x)^{2h-1} + (x^3+x)^{2h-2} +\cdots$\\
So, $\alpha_3(M^{2h-1})$ (resp. $\alpha_1(M^{2h-1})$) which is the coefficient of $x^{6h-6}$ (resp. of $x^{6h-4}$) in $M^{2h-1}$, equals $1$ (resp. $0$).
\end{proof}

\subsubsection{Case where $7$ divides $2h+1$}

In this case, by Lemma \ref{p-reduction}, $\sigma(M^6)$ divides $\sigma(M^{2h})$, where $\sigma(M^6) = (x^3+x^2+1)(x^6+x^5+1)(x^9+x^7+x^5+x+1)$ is divisible by the non Mersenne prime
$x^9+x^7+x^5+x+1 = 1+x(x+1)^2(x^3+x+1)^2$.

\subsubsection{Case where $3$ is the unique prime factor of $2h+1$}
In this case, $2h+1 = 3^w$, with $w\geq 2$ because $2h+1 \geq 9$. So, $9$ divides $2h+1$ and thus $\sigma(M^8)$ divides $\sigma(M^{2h})$ (by Lemma \ref{p-reduction}). We are done because
$\sigma(M^8) = (x^2+x+1)(x^4+x^3+1)(x^6+x+1)(x^{12}+x^8+x^7+x^4+1)$, where $x^6+x+1 = 1+x(x+1)M_3$ is not a Mersenne prime.

\subsection{Case where $M \not\in {\mathcal{M}}$ and $2h+1$ is divisible by a Mersenne prime number $p \not= 7$} \label{mersnumber}
Set $p:=2^m-1$, where $m$ and $p$ are both prime. We shall prove that $\sigma(M^{p-1})$ is divisible by a non Mersenne prime. Note that there are (at present)
``only'' 51 known Mersenne prime numbers (OEIS Sequences A$000043$ and A$000668$). The first five of them are: $3,7,31, 127$ and $8191$.

Here, $a+b = \deg(M) \geq 5$ since $M \not\in {\mathcal{M}}$.
Corollary \ref{anydivides} and Lemma \ref{phiandirreduc}-iii) imply that for $p\geq 31$, we get our result.
It remains then the case $p=3$ because $p \not= 7$.

Lemma \ref{oldresult3} has already treated the case where $\omega(\sigma(M^{2}))=2$. So, we suppose that $\omega(\sigma(M^{2})) \geq 3$. Put:
$$\sigma(M^{2}) = M_1 \cdots M_r, \ r \geq 3 \text{ and } W:=U_4=\sigma(\sigma(M^2)).$$
We get by Corollary \ref{reduction4}:

\begin{lemma} \label{smalldivisors}
i) $1+x+x^2$ divides $\sigma(M^2)$.\\
ii) No irreducible polynomial of degree $r \geq 3$ such that $2^r-1$ is prime divides $\sigma(M^2)$.
\end{lemma}

\begin{lemma} \label{lesalphasigmaM2}
Write $\sigma(M^2) = M_1 B$ where $M_1=1+x+x^2$, $\gcd(M_1,B) = 1$. One has:
$$\begin{array}{l}
i) \ \alpha_1(\sigma(M^2)) = \alpha_1(B)+1,\ \alpha_2(\sigma(M^2)) = \alpha_2(B)+\alpha_1(B)+1,\\
ii) \ \alpha_3(\sigma(M^2)) = \alpha_3(B)+\alpha_2(B)+\alpha_1(B),\\
iii) \ \alpha_3(\sigma(M^2)) = 0.

\end{array}$$
\end{lemma}
\begin{proof}
$\sigma(M^2)=M_1B = (x^2+x+1)B$. So we directly get i) and ii).\\
iii): $\sigma(M^2)=1+M+M^2=x^{2a}(x+1)^{2b} + x^a(x+1)^b+1$.\\
$2a+2b-3 > a+b$ because $a+b \geq 4$ and $x^{2a}(x+1)^{2b}$ is a square. So, $\alpha_3(\sigma(M^2))=\alpha_3(x^{2a}(x+1)^{2b})=0$.\\
\end{proof}

\begin{lemma} \label{lesalphaW}
One has:
$$\alpha_1(W) = \alpha_1(B)+1,\ \alpha_2(W)=\alpha_2(B) + \alpha_1(B),\ \alpha_3(W)=\alpha_3(B) + \alpha_2(B).$$
\end{lemma}
\begin{proof}
$W=\sigma(\sigma(M^2))=\sigma(M_1B)=\sigma(M_1)\sigma(B)=(x^2+x) \sigma(B)$.\\
Moreover, any irreducible divisor of $B$ has degree more than $3$. Hence, $\alpha_l(\sigma(B)) = \alpha_l(B),$ for $1\leq l \leq 3$. One gets:
$$\begin{array}{l}
\alpha_1(W)=\alpha_1(\sigma(B))+1=\alpha_1(B)+1,\\ 
\alpha_2(W)=\alpha_2(\sigma(B))+\alpha_1(\sigma(B))= \alpha_2(B)+\alpha_1(B),
\end{array}$$
and $\alpha_3(W)=\alpha_3(\sigma(B))+\alpha_2(\sigma(B)) = \alpha_3(B) + \alpha_2(B)$.
\end{proof}
Corollary below contradicts the fact that $W$ is a square and finishes the proof for $p=3$.
\begin{corollary}
$\alpha_3(W) = 1$.
\end{corollary}
\begin{proof}
$W$ is a square, so $0=\alpha_1(W)=\alpha_1(B)+1$ and thus $\alpha_1(B)=1$.\\
Lemma \ref{lesalphasigmaM2}-iii) implies that $0=\alpha_3(\sigma(M^2)) = \alpha_3(B)+\alpha_2(B)+\alpha_1(B)$. Therefore, we get: $\alpha_3(W)= \alpha_3(B)+\alpha_2(B) = \alpha_1(B)=1.$
\end{proof}
\begin{remark}
\emph{Our method fails for  $p=7$. Indeed, for many $M$'s, one has $\alpha_3(W) = \alpha_5(W) =0$ so that we do not reach a contradiction. We should find a large enough odd integer $l$ such that $\alpha_l(W) =0$. But, this does not appear always possible.}
\end{remark}

\subsection{Case where $M \not\in {\mathcal{M}}$ and $2h+1$ is divisible by a prime $p$ with $ord_p(2) \equiv 0 \mod 8$} \label{case-p-Fermat}
Lemmas \ref{mersennedeg8k} and \ref{ord(2)divise} imply Corollary \ref{pFermat}.
\begin{lemma} \label{mersennedeg8k}
There exists no Mersenne prime of degree multiple of $8$.
\end{lemma}

\begin{proof}
If $Q=1+x^{c_1}(x+1)^{c_2}$ with $c_1+c_2 = 8k$, then $\omega(Q)$ is even by \cite[Corollary 3.3]{Gall-Rahav-mersenn}.
\end{proof}

\begin{corollary} \label{pFermat}
If $2h+1$ is divisible by a prime $p$ such that $8$ divides $ord_p(2)$ $($in particular, when $p>5$ is a Fermat prime$)$, then $\sigma(M^{2h})$ is divisible by a non Mersenne prime.
\end{corollary}
\begin{proof}
If not, Lemma \ref{ord(2)divise} implies that $ord_p(2)$ divides $\deg(P_j)$, for any $j \in J$. So, we get a contradiction to Lemma \ref{mersennedeg8k}: $8$ divides $\deg(P_j)$.\\
In particular, if $p=2^{2^w}+1$, with $w \geq 2$, then $ord_p(2) = 2^{w+1}$ which is divisible by $8$.
\end{proof}

\begin{remarks}
\emph{i) If $p$ is a Fermat prime, then $ord_p(2) \equiv 0 \mod 8$. The converse is false. Examples: $p \in \{97,673\}$ with $ord_p(2)=48$.}\\
\emph{ii) It remains the following (large) case to complete the proof of Conjecture \ref{oldconject}:
$M \not\in \mathcal{M}$ and $2h+1$ is divisible by $p \in \{5,7\}$ or by $p>7$ which is neither Mersenne prime nor Fermat prime.\\
Moreover, assuming Conjecture \ref{oldconject}, similar proofs as in Section \ref{case-perfect} would state that there exists no odd perfect polynomial over $\F_2$ which is only divisible by Mersenne primes.}
\end{remarks}
\def\biblio{\def\titrebibliographie{References}\thebibliography}
\let\endbiblio=\endthebibliography




\newbox\auteurbox
\newbox\titrebox
\newbox\titrelbox
\newbox\editeurbox
\newbox\anneebox
\newbox\anneelbox
\newbox\journalbox
\newbox\volumebox
\newbox\pagesbox
\newbox\diversbox
\newbox\collectionbox
\def\fabriquebox#1#2{\par\egroup
\setbox#1=\vbox\bgroup \leftskip=0pt \hsize=\maxdimen \noindent#2}
\def\bibref#1{\bibitem{#1}


\setbox0=\vbox\bgroup}
\def\auteur{\fabriquebox\auteurbox\styleauteur}
\def\titre{\fabriquebox\titrebox\styletitre}
\def\titrelivre{\fabriquebox\titrelbox\styletitrelivre}
\def\editeur{\fabriquebox\editeurbox\styleediteur}

\def\journal{\fabriquebox\journalbox\stylejournal}

\def\volume{\fabriquebox\volumebox\stylevolume}
\def\collection{\fabriquebox\collectionbox\stylecollection}
{\catcode`\- =\active\gdef\annee{\fabriquebox\anneebox\catcode`\-
=\active\def -{\hbox{\rm
\string-\string-}}\styleannee\ignorespaces}}
{\catcode`\-
=\active\gdef\anneelivre{\fabriquebox\anneelbox\catcode`\-=
\active\def-{\hbox{\rm \string-\string-}}\styleanneelivre}}
{\catcode`\-=\active\gdef\pages{\fabriquebox\pagesbox\catcode`\-
=\active\def -{\hbox{\rm\string-\string-}}\stylepages}}
{\catcode`\-
=\active\gdef\divers{\fabriquebox\diversbox\catcode`\-=\active
\def-{\hbox{\rm\string-\string-}}\rm}}
\def\ajoutref#1{\setbox0=\vbox{\unvbox#1\global\setbox1=
\lastbox}\unhbox1 \unskip\unskip\unpenalty}
\newif\ifpreviousitem
\global\previousitemfalse
\def\separateur{\ifpreviousitem {,\ }\fi}
\def\voidallboxes
{\setbox0=\box\auteurbox \setbox0=\box\titrebox
\setbox0=\box\titrelbox \setbox0=\box\editeurbox
\setbox0=\box\anneebox \setbox0=\box\anneelbox
\setbox0=\box\journalbox \setbox0=\box\volumebox
\setbox0=\box\pagesbox \setbox0=\box\diversbox
\setbox0=\box\collectionbox \setbox0=\null}
\def\fabriquelivre
{\ifdim\ht\auteurbox>0pt
\ajoutref\auteurbox\global\previousitemtrue\fi
\ifdim\ht\titrelbox>0pt
\separateur\ajoutref\titrelbox\global\previousitemtrue\fi
\ifdim\ht\collectionbox>0pt
\separateur\ajoutref\collectionbox\global\previousitemtrue\fi
\ifdim\ht\editeurbox>0pt
\separateur\ajoutref\editeurbox\global\previousitemtrue\fi
\ifdim\ht\anneelbox>0pt \separateur \ajoutref\anneelbox
\fi\global\previousitemfalse}
\def\fabriquearticle
{\ifdim\ht\auteurbox>0pt        \ajoutref\auteurbox
\global\previousitemtrue\fi \ifdim\ht\titrebox>0pt
\separateur\ajoutref\titrebox\global\previousitemtrue\fi
\ifdim\ht\titrelbox>0pt \separateur{\rm in}\
\ajoutref\titrelbox\global \previousitemtrue\fi
\ifdim\ht\journalbox>0pt \separateur
\ajoutref\journalbox\global\previousitemtrue\fi
\ifdim\ht\volumebox>0pt \ \ajoutref\volumebox\fi
\ifdim\ht\anneebox>0pt  \ {\rm(}\ajoutref\anneebox \rm)\fi
\ifdim\ht\pagesbox>0pt
\separateur\ajoutref\pagesbox\fi\global\previousitemfalse}
\def\fabriquedivers
{\ifdim\ht\auteurbox>0pt
\ajoutref\auteurbox\global\previousitemtrue\fi
\ifdim\ht\diversbox>0pt \separateur\ajoutref\diversbox\fi}
\def\endbibref
{\egroup \ifdim\ht\journalbox>0pt \fabriquearticle
\else\ifdim\ht\editeurbox>0pt \fabriquelivre
\else\ifdim\ht\diversbox>0pt \fabriquedivers \fi\fi\fi.\voidallboxes}

\let\styleauteur=\sc
\let\styletitre=\it
\let\styletitrelivre=\sl
\let\stylejournal=\rm
\let\stylevolume=\bf
\let\styleannee=\rm
\let\stylepages=\rm
\let\stylecollection=\rm
\let\styleediteur=\rm
\let\styleanneelivre=\rm

\begin{biblio}{99}

\begin{bibref}{Beard2}
\auteur{J. T. B. Beard Jr}  \titre{Perfect polynomials revisited}
\journal{Publ. Math. Debrecen} \volume{38/1-2} \pages 5-12 \annee
1991
\end{bibref}

\begin{bibref}{Beard}
\auteur{J. T. B. Beard Jr, J. R. Oconnell Jr, K. I. West}
\titre{Perfect polynomials over $GF(q)$} \journal{Rend. Accad.
Lincei} \volume{62} \pages 283-291 \annee 1977
\end{bibref}

\begin{bibref}{Canaday}
\auteur{E. F. Canaday} \titre{The sum of the divisors of a
polynomial} \journal{Duke Math. J.} \volume{8} \pages 721-737 \annee
1941
\end{bibref}

\begin{bibref}{Gall-Rahav4}
\auteur{L. H. Gallardo, O. Rahavandrainy} \titre{Odd perfect
polynomials over $\F_2$} \journal{J. Th\'eor. Nombres Bordeaux}
\volume{19} \pages 165-174 \annee 2007
\end{bibref}

\begin{bibref}{Gall-Rahav7}
\auteur{L. H. Gallardo, O. Rahavandrainy} \titre{There is no odd
perfect polynomial over $\F_2$ with four prime factors}
\journal{Port. Math. (N.S.)} \volume{66(2)} \pages 131-145 \annee
2009
\end{bibref}

\begin{bibref}{Gall-Rahav5}
\auteur{L. H. Gallardo, O. Rahavandrainy} \titre{Even perfect
polynomials over $\F_2$ with four prime factors} \journal{Intern. J.
of Pure and Applied Math.} \volume{52(2)} \pages 301-314 \annee 2009
\end{bibref}

\begin{bibref}{Gall-Rahav8}
\auteur{L. H. Gallardo, O. Rahavandrainy} \titre{All perfect
polynomials with up to four prime factors over $\F_4$}
\journal{Math. Commun.} \volume{14(1)} \pages 47-65 \annee 2009
\end{bibref}

\begin{bibref}{Gall-Rahav2}
\auteur{L. H. Gallardo, O. Rahavandrainy} \titre{On splitting
perfect polynomials over $\F_{p^p}$} \journal{Int. Electron. J.
Algebra} \volume{9} \pages 85-102 \annee 2011
\end{bibref}

\begin{bibref}{Gall-Rahav12}
\auteur{L. H. Gallardo, O. Rahavandrainy} \titre{On even (unitary) perfect
polynomials over $\F_{2}$ } \journal{Finite Fields Appl.} \volume{18} \pages 920-932 \annee 2012
\end{bibref}

\begin{bibref}{Gall-Rahav13}
\auteur{L. H. Gallardo, O. Rahavandrainy} \titre{Characterization of Sporadic perfect
polynomials over $\F_{2}$ } \journal{Functiones et Approx.} \volume{55.1} \pages 7-21 \annee 2016
\end{bibref}

\begin{bibref}{Gall-Rahav-mersenn}
\auteur{L. H. Gallardo, O. Rahavandrainy} \titre{On Mersenne
polynomials over $\F_{2}$} \journal{Finite Fields Appl.}
\volume{59} \pages 284-296 \annee 2019
\end{bibref}

\begin{bibref}{Rudolf}
\auteur{R. Lidl, H. Niederreiter} \titrelivre{Finite Fields,
Encyclopedia of Mathematics and its applications} \editeur{Cambridge
University Press} \anneelivre 1983 (Reprinted 1987)
\end{bibref}

\end{biblio}
\end{document}